\documentclass[draft|showlink]{siamart1116}

\usepackage{amsfonts}
\usepackage{graphicx}
\usepackage{epstopdf}
\usepackage{algorithmic}
\ifpdf
\DeclareGraphicsExtensions{.eps,.pdf,.png,.jpg}
\else
\DeclareGraphicsExtensions{.eps}
\fi

\newtheorem{assumption}{Assumption}

\newcommand{\TheTitleABB}{A numerical scheme for high-dimensional BSDES based on modified}
\newcommand{\TheTitle}{A numerical scheme for high-dimensional backward stochastic differential equation based on modified multi-level Picard iteration}
\newcommand{\TheAuthors}{Chol-Kyu Pak, Mun-Chol Kim and Hun O}

\def\to{\rightarrow}
\def\beq{\begin{equation}} \def\enq{\end{equation}}
\def\beseq{\begin{subequations}}  \def\enseq{\end{subequations}}
\def\beqa{\begin{eqnarray}} \def\enqa{\end{eqnarray}}

\def\BeDef{\begin{definition}} \def\EnDef{\end{definition}}
\def\BeThe{\begin{theorem}} \def\EnThe{\end{theorem}}
\def\BeLem{\begin{lem}} \def\EnLem{\end{lem}}

\numberwithin{theorem}{section}
\numberwithin{equation}{section}
\numberwithin{table}{section}
\numberwithin{figure}{section}

\headers{\TheTitleABB}{\TheAuthors}

\title{{\TheTitle}}

\author{
	Chol-Kyu Pak\thanks{Faculty of Mathematics, Kim Il Sung University, Pyongyang, Democratic People's Republic of Korea (\email{pck2016217@gmail.com}).}
	\and
	Mun-Chol Kim
	\and
	Hun O
}

\usepackage{amsopn}

\begin{document}
	
	\maketitle
	
	\begin{abstract}
		In this paper we propose a new kind of numerical scheme for high-dimensional backward stochastic differential equations based on modified multi$-$level Picard iteration. The proposed scheme is very similar to the original multi-level Picard iteration but it differs on underlying Monte$-$Carlo sample generation and enables an improvement in the sense of complexity.\\
		We prove the explicit error estimates for the case where the generator does not depend on control variate.\\
		\textbf{Keywords}  backward stochastic differential equations; numerical scheme; curse-of-dimensionality; error estimates
	\end{abstract}
	
	\section{Introduction}
	
	Let $\left(\Omega, \mathcal{F}, P\right)$ be a probability space, $T>0$ a finite time and $\left\{\mathcal{F}_t \right\}_{0\leq t\leq T}$ a filtration satisfying the usual conditions. Let $\left(\Omega, \mathcal{F}, P,\left\{\mathcal{F}_t \right\}_{0\leq t\leq T}\right)$ be a complete, filtered probability space on which a standard $d$-dimensional Brownian motion $W_t=\left(W_t^1,W_t^2,\ldots,W_t^d \right)^T$ is defined and $\mathcal{F}_0$ contains all the $P-$null sets of $\mathcal{F}$. Let $L^2=L_{\mathcal{F}}^2 (0,T)$ be the set of all $\{\mathcal{F}_t\}-$adapted mean$-$square$-$integrable processes.\\
	We consider the backward stochastic differential equation (BSDE)\\
	
	\begin{eqnarray}\label{1.1}
	y_t=\xi+\int\limits_{t}^T f(s, y_s,z_s)ds-\int\limits_{t}^T z_sdWs, \, t\in [0,T]
	\end{eqnarray}
		
	\noindent
	where the generator $f=f(t,y,z)$ is a vector function valued in $\mathbb{R}^m$  and is $\mathcal{F}_t-$adapted for each $(y,z)$ and the terminal variable $\xi \in L^2$ is $\mathcal{F}_T-$  measurable.\\
	A process $(y_t,z_t ):[0,T]\times \Omega \to \mathbb{R}^m\times \mathbb{R}^{m\times d}$ is called an $L^2-$solution of the BSDE (\ref{1.1}) if it is $\{\mathcal{F}_t\}-$adapted, square integrable, and satisfies the equation.	
	In 1990, Pardoux and Peng first proved in \cite{6} the existence and uniqueness of the solution of general nonlinear BSDEs and afterwards there has been very active research in this field with many applications.(\cite{5,7,9,10} )\\
	In this paper we assume that the terminal condition is a function of $W_T$, i.e. $\xi=\varphi(W_T)$ and the BSDE (\ref{1.1}) has a unique solution $(y_t,z_t )$.\\
	It was shown in \cite{7} that the solution $(y_t,z_t )$ of (\ref{1.1}) can be represented as
	
	\begin{eqnarray}\label{1.2}
	y_t=u(t,W_t), \, z_t=\nabla_x u(t,W_t ), \,  \forall t\in[0,T)  	\end{eqnarray}
	where $u(t,x)$  is the solution of the following parabolic partial differential equation
	
	\begin{eqnarray}\label{1.3}
	\dfrac{\partial u}{\partial t}+\dfrac{1}{2}\sum_{i=1}^d \dfrac{\partial^2 u}{\partial x_i^2}+f(t,u, \nabla_x u)=0 \end{eqnarray}
	
	\noindent
	with the terminal condition $u(T,x)=\varphi(x)$, and $\nabla_x u$ is the gradient of $u$ with respect to the spacial variable $x$. The smoothness of $u$ depends on $f$ and $\varphi$.\\
	
	Furthermore, the celebrated Feynman-Kac formula holds as follows.\\
	
	\begin{eqnarray}
	\nonumber u(t,x)&=&E\left[\varphi\left(x+W_{T-t}\right)\right]\\
	&&+\int\limits_{t}^TE\left[f(s, u(s,x+W_{s-t}),\nabla u(s,x+W_{s-t}))\right]ds \label{1.4}\end{eqnarray}

\noindent	
	Moreover, regarding to the gradient of solution,  Bismut-Elworthy-Li formula (\cite{4}) holds as follows.
	
	\begin{eqnarray}
	\nonumber \dfrac{\partial u}{\partial x_k}(t,x)&=&E\left[\varphi\left(x+W_{T-t}\dfrac{W_{T-t}^k}{T-t}\right)\right]\\
	&&+\int\limits_{t}^TE\left[f(s, u(s,x+W_{s-t}),\nabla u(s,x+W_{s-t}))\dfrac{W_{s-t}^k}{s-t}\right]ds \label{1.5}\end{eqnarray}
	
	Although BSDEs have very important applications in many fields such as mathematical finance and stochastic control, it is well known that it is difficult to obtain analytic solutions except some special cases and there have been many works on numerical methods to get approximate solution.\\
	In \cite{5} a numerical method by using binomial approach was proposed and in \cite{9,10} a $\theta$-scheme which is based on time-space grid is proposed. Besides, there are several kinds of numerical methods for BSDEs (see references in \cite{2}).\\
	All the algorithms mentioned above suffer from curse-of-dimensionality, i.e., the computational complexity grows exponentially in dimensionality. Most of the problems from real applications are high-dimensional and it has been very challenging to build a scheme of which complexity grows polynomially in dimension.\\
	\par Recently Martin et al. \cite{1,2} proposed such a scheme for the first time based on Monte-Carlo simulation and multi-level Picard iteration. Their multi-level approach is very effective and it shows applications in many important numeric fields. They proved the error estimates of the multi-level Picard approximation rigorously in \cite{3}.\\
	\par In this paper, we propose a scheme which is based on modified multi-level Picard iteration. The proposed scheme is very similar to the original multi-level scheme but makes use of different underlying Monte-Carlo samples and it enables a slight improvement in complexity. We note that the improvement just lies in constant scale (approximately $1.5\sim 2$ times). But we still believe the idea should be effective and applicable to other applications.\\
\par 	The rest of this paper is organized as follows. In Section 2 we introduce the scheme based on original multi-level Picard iteration and present a new scheme. In Section 3, we give error estimates of the proposed scheme theoretically under some assumptions for a limited case. In Section 4, some conclusions are given.
	\section{Multi-level Picard iteration and modification}
	First we introduce the multi-level Picard iteration method proposed in \cite{1}.
	
	\begin{itemize}
		\item[-] Initial step ($n=0$)\\
	\end{itemize}
	\begin{eqnarray}\label{2.1}
	y_0(t,x)=0, z_0(t,x)=0
	\end{eqnarray}
	\begin{itemize}	
		\item[-] Iteration step ($n\geq 1$)
	\end{itemize}
	\begin{align}\label{2.2}
	y_n(t,x)&=\dfrac{1}{M^n}\sum_{i=1}^{M^n} \varphi(t,x+W_{T-t}^{(n,i,0)})&
	\\ \nonumber+&\sum_{l=0}^{n-1}\sum_{j=1}^Q \dfrac{W_{n,j}}{M^{n-1}}\sum_{i=1}^{M^{n-l}}\left[f(y_l,z_l)-\chi(l\neq 0)f(y_{l-1},z_{l-1})(t_{n,j,x}+W_{t_{n,j}-t}^{(n,i,1)})\right]
	\\\label{2.3} z_n(t,x)&=\dfrac{1}{M^n}\sum_{i=1}^{M^n} \left(\varphi(t,x+W_{T-t}^{(n,i,0)})-\varphi(t,x)\right)\dfrac{W_{T-t}^{(n,i,0)}}{T-t}\\
	\nonumber +&\sum_{l=0}^{n-1}\sum_{j=1}^Q \dfrac{W_{n,j}}{M^{n-1}}\sum_{i=1}^{M^{n-l}}\left[f(y_l,z_l)-\chi(l\neq 0)f(y_{l-1},z_{l-1})(t_{n,j,x}+W_{t_{n,j}-t}^{(n,i,1)})\dfrac{W_{t_{n,j}-t}^{(n,i,1)}}{t_{n,j}-t}\right]
	\end{align}

	In the above scheme, $W_{s-t}^{(n,i)}$ denotes the $i$-th $d$-dimensional Gauss random variable that is used for Monte-Carlo approximation of the expectation at $n$-th iteration stage and $\left\{W_{s-t}^{(n,i)}\right\}$ are mutually independent.\\
	
	Gauss-Legendre quadrature with $Q-$points is used for the  approximation of the time integral. $(t_{n,1},\ldots,t_{n,Q} )\in [t,T]$ are the quadrature points and $w_{n,j}$ denotes the weight for the point $t_{n,j}$.\\
	
	We propose a modified multi-level Picard approximation scheme as follows.
	
	\begin{itemize}
		\item[-] Initial step ($n<2$)\\
	\end{itemize}
		\begin{align}\label{2.4}
		y_0(t,x)=0, z_0(t,x)=0\\
		\nonumber y_1(t,x)=\dfrac{1}{M}\sum_{i=1}^M\varphi\left(x+W_{T-t}^{(1,i)}\right)+\sum_{i=1}^Qw_{1,j}f\left(t_{1,j},0,0\right)\\
		 \nonumber z_1(t,x)=\dfrac{1}{M}\sum_{i=1}^M\left(\varphi\left(x+W_{T-t}^{(1,i)}\right)-\varphi(x)\right)\dfrac{W_{T-t}^{(1,i)}}{T-t}+\\
		+\sum_{i=1}^Q\dfrac{w_{1,j}}{M}\sum_{i=0}^M f\left(t_{1,j},0,0\right)\dfrac{W_{t_{1,j}-t}^{(1,i)}}{t_{1,j}-t} \label{2.5}
		\end{align}
		\begin{itemize}
		\item[-] Iteration step ($n\geq 2$)
		\end{itemize}
		\begin{eqnarray}
		\nonumber &&y_n(t,x)=\dfrac{1}{M}\sum_{i=1}^My_{n-1}^i(t,x)\\&& \nonumber +\sum_{j=1}^Q\dfrac{w_{n,j}}{M}\sum_{i=1}^M\left[\left(f\left(y_{n-1},z_{n-1}\right)-f(y_{n-2},z_{n-2})\right)\left(t_{n,j},x+W_{t_{n,j}-t}^{(n,i)}\right)\right]\\\label{2.6}
		\end{eqnarray}
		\begin{eqnarray}
		\nonumber && z_n(t,x)=\dfrac{1}{M}\sum_{i=1}^Mz_{n-1}^i(t,x)\dfrac{W_{T-t}^{(n,i)}}{T-t}\\
		\nonumber &&+\sum_{j=1}^Q\dfrac{w_{n,j}}{M}\sum_{i=1}^M\left[\left(f\left(y_{n-1},z_{n-1}\right)-f(y_{n-2},z_{n-2})\right)\left(t_{n,j},x+W_{t_{n,j}-t}^{(n,i)}\right)\dfrac{W_{t_{n,j}-t}^{(n,i)}}{t_{n,j}-t}\right]\\\label{2.7}
		\end{eqnarray}
	
	In the above scheme, $\{y_n^i (t,x)\}$ are the i.i.d. random variables that is identically distributed with $y_n (t,x)$.\\
	
	The proposed scheme (\ref{2.4})-(\ref{2.7}) is very similar to the original multi-level Picard iteration (\ref{2.1})-(\ref{2.3}).	
	Actually if we expand $y_n (t,x)$ one step further, we have the following.
	\begin{eqnarray*}
		y_n(t,x)&=&\dfrac{1}{M^2}\sum_{i=1}^{M^2}y_{n-2}^i(t,x)\\
		&&+\sum_{j=1}^Q\dfrac{w_{n,j}}{M}\sum_{i=1}^M\left[\left( f(y_{n-1},z_{n-1})-f(y_{n-2},z_{n-2})\right)\left(t_{n,j,x}+W_{t_{n,j}-t}^{(n,i)}\right)\right]\\
		&&+\sum_{j=1}^Q\dfrac{w_{n-1,j}}{M^2}\sum_{i=1}^{M^2}\left[\left( f(y_{n-2},z_{n-2})-f(y_{n-3},z_{n-3})\right)\left(t_{n-1,j,x}+W_{t_{n-1,j}-t}^{(n-1,i)}\right)\right]
	\end{eqnarray*}
	
	Likewise if we repeatedly expand we have a very similar one to (\ref{2.1})-(\ref{2.3}). But note that the proposed scheme makes use of Monte-Carlo samples of different level to calculate $y_n (t,x)$ while (\ref{2.2})-(\ref{2.3}) uses the ones of the same level. More importantly, we need to calculate $y_{n-1} (t,x)$ to get $y_n (t,x)$. So for the evaluation of $f(y_{n-2},z_{n-2})$ at $\left(t_{n,j},x+W_{t_{n,j}-t}^{(n,i)}\right)$ on the right hand of (\ref{2.6})-(\ref{2.7}) we can use the one that we used for evaluation of $f(y_{n-1},z_{n-1})$ at the same time-space point. This reduces the computational complexity slightly, to half roughly.\\
	\section{Error estimates}

	In this section, we prove the convergence of the proposed scheme for a limited case where the driver $f(s,y,z)$ does not depend on control variate. We also note that we just present the proof for $y$ component.\\
	
	First let us point out some results on Gauss-Legendre quadrature.\\
	
	Let $q\in \mathbb{N}$ , let $(C_j^q)^{j-1, q}\subset [-1,1]$ be the roots of  $q$-th order Legendre polynomial
	
	$$L^q(x)=\dfrac{1}{2^qq!}\dfrac{d^q}{dx^q}[(x^2-1)^q]$$
	.
	For the real function $g(s):[a,b]\to \mathbb{R}$, its integral $\int\limits_a^b g(t)dt$  can be approximated using $q$  points Gauss-Legendre quadrature as follows.\\
	
	\begin{eqnarray}\label{3.1}
	\int\limits_a^b g(t)dt-\sum_{j=1}^q w_jg(t_j)=\dfrac{\left[q!\right]^4(b-a)^{2q+1}}{(2q+1)\left[(2q)!\right]^3}g^{(2q)
	}(\xi), \, \xi\in[a,b]
	\end{eqnarray}
	
	where $(t_j)_{j=1\cdots q}$  are the quadrature points and $(w_j)_{j=1\cdots q}$  are their corresponding weights defined as follows.\\
	
	$$t_j=\dfrac{c_j^q(b-a)+(a+b)}{2}, \, w_j=\int\limits_a^b\left[\prod_{{i\in\{1,\ldots,q\},{t_i\neq t_j}}}\dfrac{2x-(b-a)c_i^q-(a+b)}{2t_j-(b-a)c_i^q-(a+b)}\right]dx$$
	
	For the simplicity, in the rest of the paper we denote the $q$-points Gauss-Legendre quadrature approximation of $\int\limits_a^bg(t)dt$ by $\int\limits_{[a,b],q}^{\sim}g(t)dt$ . Moreover as long as the number of quadrature points does not change, we simply write as  $\int\limits_{[a,b]}^{\sim}g(t)dt$.
	
	$$\int\limits_{[a,b],q}^{\sim}g(t)dt=\sum_{j=1}^q w_jg(t_j)$$
	
	The following proposition is a direct result from (\ref{3.1}) and it will be used in many places later.
	
	\begin{proposition}\label{prop3.1}
		For a sufficiently smooth real function $g:[a,b]\to \mathbb{R}$ which satisfies
		$$\forall n\in \mathbb{N}, \forall t\in [0,T], g^{(2n)}(t)\geq 0$$
		
		the following inequality holds

		$$\forall q \in \mathbb{N}, \int\limits_{[a,b],q}^{\sim}g(t)dt=\sum_{j=1}^q w_jg(t_j)\leq \int\limits_a^b g(t)dt,$$
		
		\noindent where $\int\limits_{[a,b],q}^{\sim}g(t)dt=\sum_{j=1}^q w_jg(t_j)$ denotes the  $q$-points Gauss-Legendre quadrature approximation of $\int\limits_a^b g(t)dt$.\\
	\end{proposition}

	Now we make it clear the meaning of  $y_n(t,x), z_n(t,x)$.\\
	For $n=1$ , we have
	$$y_1(t,x)=\dfrac{1}{M}\sum_{i=1}^M \varphi\left(x+W_{T-t}^{(1,i)}\right)+\sum_{j=1}^Q w_{1,j}f(t_{1,j},0)$$
	
	\noindent and $y_1(t,x)$  is the random variable in the probability space $$\left(\Omega_1, \mathcal{F}_1, P_1\right)=\left(\Omega, \mathcal{F}, P\right).$$
	For $n=2$ , we have
	$$y_2(t,x)=\dfrac{1}{M}\sum_{i=1}^M y_1^i(t,x)+\sum_{j=1}^Q \dfrac{w_{2,j}}{M}\sum_{i=1}^M\left[f\left(t_{2,j},y_1\left(W_{t_{2,j}-t}^{(2,i)}\right)\right)-f(t_{2,j},0)\right]$$
	\noindent and the right hand includes $y_1(t_{2,j},x+W_{t_{2,j}-t}^{(2,i)})$ . So $y_2(t,x)$  is the random variable in the probability space$\left(\Omega_2, \mathcal{F}_2, P_2\right)=\left(\Omega\times \Omega, \mathcal{F}\otimes \mathcal{F}, P\times P\right)$ which depends on $W^{(2,i)}$ and $W^{(1,i)}$.\\
	
	Likewise for the  $n$-th iteration stage, $y_n(t,x)$  is the random variable in the probability space $\left(\Omega_n, \mathcal{F}_n, P_n\right)=\left(\underbrace{\Omega\times \Omega\times\cdots\times \Omega}_{n}, \underbrace{\mathcal{F}\otimes \mathcal{F}\otimes\cdots\otimes \mathcal{F}}_{n}, \underbrace{P\times P\times\cdot\times P}_{n}\right)$ . We denote the expectation, variance and $L_p$-norm in $\left(\Omega_n, \mathcal{F}_n, P_n\right)$   by $E^n [\cdot]$, $var^n [\cdot]$  and $\|\|_{n,p}$ respectively.
	
	The following propositions will be used in the error estimates
	
	\begin{proposition}\cite{8}\label{prop3.2} $$\forall n\in \mathbb{N}, \sqrt{2\pi n}\left(\dfrac{n}{e}\right)^n\leq n!\leq \sqrt{2\pi n}\left(\dfrac{n}{e}\right)^n e^{\frac{1}{12}}$$
	\end{proposition}
	
	\begin{proposition}\label{prop3.3}
		$$\forall n\in \mathbb{N}, k<n, \left(
		\begin{array}{c}
		n \\
		k
		\end{array}
		\right)<2^n
		$$
	\end{proposition}
	
	Now we address some necessary assumptions for the error estimates.
	\begin{assumption}\label{ass1}
		The generator $f(s,y)$  is Lipschitz continuous in $y$  and $f(s,0)$  is globally bounded.
		\begin{eqnarray*}
			\exists C_f>0, \forall y_1,y_2 \in \mathbb{R}, \forall s\in [0,T], |f(s,y_1)-f(s,y_2)|\leq C_f|y_1-y_2|\\
			\exists C_0>0, \forall s\in[0,T], |f(s,0)|\leq C_0
		\end{eqnarray*}
	\end{assumption}
	
	\begin{assumption}\label{ass2} The terminal condition $\varphi(x)$ is bounded.
		
		$$\exists C_{\varphi}>0, \forall x\in \mathbb{R}^d, \, |\varphi(x)|\leq C_{\varphi}$$
	\end{assumption}
	
	\begin{assumption}\label{ass3}
		The true solution $y(t,x)$ is bounded
		
		$$\exists C_{y}>0,\forall t\in [0, T], \forall x\in \mathbb{R}^d, \, |y(t,x)|\leq C_{y}$$
	\end{assumption}
	
	\begin{assumption}\label{ass4}
		For any $t\in [0,T], x\in \mathbb{R}^d$ , if we define two real functions as follows		
		$$F_t(s):=E\left[f\left(s, y\left(s,x+ W_{s-t}\right)\right)\right], \, G_t(s):=E\left[f\left(s, y\left(s,x+ W_{s-t}\right)\right)\dfrac{W_{s-t}}{s-t}\right] $$
		,
		then $F_t (s),G_t (s)$ are bounded, smooth enough and all of their derivatives are also bounded.
		
		$$\exists C_d>0, \forall t\in [0,T], x\in \mathbb{R}^d,  \sup_{s\in [t,T],k\in \mathbb{N}}F_t^{(k)}(s)\leq C_d, \, \sup_{s\in [t,T],k\in \mathbb{N}}G_t^{(k)}(s)\leq C_d $$
	\end{assumption}
	
	It is not difficult to check that the Assumption (\ref{ass4}) holds if the generator and the true solution are smooth enough, bounded and all of their derivatives are also bounded. (See \cite{10})
	
	Now we state the error estimates of the proposed scheme (\ref{2.4})-(\ref{2.7}).
	
	\begin{theorem}\label{the3.1}
		Let $(y(t, W_t),z(t,W_t))$   be the solution of the following backward stochastic differential equation.
		$$y_t=\xi+\int\limits_t^T f(s,y_s)ds-\int\limits_t^T z_sdWs, \, t\in [0,T]$$
		Let $(y_n,z_n)$  be the approximation series by the scheme (\ref{2.4})-(\ref{2.7}).
		Under the Assumptions (\ref{ass1})-(\ref{ass4}), the following estimate holds.		
		\begin{eqnarray*}
			\left|(y- E^n[y_n])(t,x)\right|&\leq& nC_2C_dQ^{\frac{1}{2}}\left(\dfrac{e}{8Q}\right)^{2Q}\\
			&+&\left(\dfrac{C_1}{\sqrt{M}}\right)^ne^{C_f\sqrt{M}(T-t)(1+1/C_1)}+\dfrac{C_y(T-t)^nC^n_f}{n!}
		\end{eqnarray*}
		
		$$\left(var^n[y_n(t,x)]\right)^{\frac{1}{2}}=\left(E^n\left[(y_n-E^ny_n)^2(t,x)\right]\right)^{\frac{1}{2}}\leq \left(\dfrac{C_1}{\sqrt{M}}\right)e^{C_f \sqrt{M}(T-t)}$$
		\noindent where $C_1,C_2$ are some constants that depend only on $T, C_f, C_y, C_0, C_d, C_{\varphi}$.
	\end{theorem}
	\begin{proof}
		 For $n\geq 2$ from (3.2) it holds that
		\begin{align}
		\nonumber E^n[y_n(t,x)]&=E^n[y_{n-1}(t,x)]\\
		\nonumber &+\sum_{j=1}^Q w_{n,j}E^n[f(t_{n,j},y_{n-1}(t_{n,j},x+W_{t_{n,j}-t}^{(n)}))-f(t_{n,j},y_{n-2}(t_{n,j},x+W_{t_{n,j}-t}^{(n)}))]\\
		\nonumber &=E^{n-1}[y_{n-1}(t,x)]+\sum_{j=1}^Q w_{n,j}E^n[f(t_{n,j},y_{n-1}(t_{n,j},x+W_{t_{n,j}-t}^{(n)}))] \\ \nonumber &-\sum_{j=1}^Qw_{n,j}E^n[f(t_{n,j},y_{n-2}(t_{n,j},x+W_{t_{n,j}-t}^{(n)}))]\\
		\nonumber &=E^{n-1}[y_{n-1}(t,x)]+\sum_{j=1}^Q w_{n,j}E^n[f(t_{n,j},y_{n-1}(t_{n,j},x+W_{t_{n,j}-t}^{(n)}))] \\\nonumber &-\sum_{j=1}^Qw_{n,j}E^{n-1}[f(t_{n,j},y_{n-2}(t_{n,j},x+W_{t_{n,j}-t}^{(n-1)}))] \\\label{3.2}
		\end{align}
		At the last step of the above equality the following reasoning was applied
		\begin{eqnarray}
		\nonumber E^n[f((t_{n,j},y_{n-2}(t_{n,j},x+W_{t_{n,j}-t}^{(n)})))]&=&\int\limits_{\mathbb{R}} E^{n-1}[f(t_{n,j},y_{n-2}(t_{n,j},x+u))]p_{t_{n,j}-t}(u)du\\
		\nonumber &=&\int\limits_{\mathbb{R}} E^{n-2}[f(t_{n,j},y_{n-2}(t_{n,j},x+u))]p_{t_{n,j}-t}(u)du\\
		&=&E^{n-1}[f((t_{n,j},y_{n-2}(t_{n,j},x+W_{t_{n,j}-t}^{(n-1)})))] \label{3.3}
		\end{eqnarray}
		\noindent where $p_{t_{n,j}-t}(u)$ denotes the probability density function of Gaussian random variable $W_{t_{n,j}-t}^{(n)}$.
		$$p_{t_{n,j}-t}(u)=\dfrac{1}{\sqrt{2\pi (t_{n,j}-t)}}\exp\left(-\dfrac{u^2}{t_{n,j}-t}\right)$$
		Repeating the similar procedure, we deduce the following result for $n\geq 2$.
		\begin{align}\label{3.4}
 E^n[y_n(t,x)]&=E^1[y_1(t,x)]+\sum_{j=1}^Qw_{n,j}E^n[f(t_{n,j},y_{n-1}(t_{n,j},x+W_{t_{n,j}-t}^{(n)}))]-\sum_{j=1}^Q w_{1,j}f(t_{1,j},0)\\
		\nonumber &= E^1[\varphi(t,x+W_{T-t}^{(1)})]+\sum_{j=1}^Qw_{n,j}E^n[f(t_{n,j},y_{n-1}(t_{n,j},x+W_{t_{n,j}-t}^{(n)}))]\\\nonumber
		&=E[\varphi(t,x+W_{T-t}^{(1)})]+\sum_{j=1}^Qw_{n,j}E^n[f(t_{n,j},y_{n-1}(t_{n,j},x+W_{t_{n,j}-t}^{(n)}))] 
		\end{align}
		Note that (\ref{3.4}) holds for  $n=1$ from (\ref{2.5}). \\
		For all $n\in \mathbb{N}$ , from (\ref{3.4}) using triangle inequality and Assumption (\ref{ass1}) it holds that
		\begin{align}\label{3.5}
		|(y-E^n[y_n])(t,x)|&=|y(t,x)-E^n[y_n(t,x)]|\\
		\nonumber &\leq \left|\int\limits_{t}^TE[f(s,y(s,x+W_{s-t}))]ds-\sum_{j=1}^Q w_{n,j}E^n[f(t_{n,j},y_{n-1}(t_{n,j},x+W_{t_{n,j}-t}^{(n)}))]\right|\\
		\nonumber &\leq \left|\int\limits_{t}^TE[f(s,y(s,x+W_{s-t}))]ds-\sum_{j=1}^Q w_{n,j}E[f(t_{n,j},y(t_{n,j},x+W_{t_{n,j}-t}))]\right|\\
		\nonumber &+\sum_{j=1}^Q w_{n,j}E^n\left[\left| f(t_{n,j},y(t_{n,j},x+W_{t_{n,j}-t}^{(n)}))-f(t_{n,j},y_{n-1}(t_{n,j},x+W_{t_{n,j}-t}^{(n)}))\right|\right]\\
		\nonumber  &\leq \epsilon_{n}(t)+C_f\sum_{j=1}^Q w_{n,j}E^n\left[\left|y(t_{n,j},x+W_{t_{n,j}-t}^{(n)})-y_{n-1}(t_{n,j},x+W_{t_{n,j}-t}^{(n)})\right|\right]\\
		\nonumber &\leq \epsilon_{n}(t)+C_f\sum_{j=1}^Q w_{n,j}E^n\left[\left|(y-E^{(n-1)}[y_{n-1}])(t_{n,j},x+W_{t_{n,j}-t}^{(n)})\right|\right]\\
		\nonumber  &+C_f\sum_{j=1}^Q w_{n,j}E^n\left[\left|(y_{n-1}-E^{(n-1)}[y_{n-1}])(t_{n,j},x+W_{t_{n,j}-t}^{(n)})\right|\right]\\
		\nonumber &\leq \epsilon_{n}(t)+C_f\sum_{j=1}^Q w_{n,j}\int\limits_{\mathbb{R}}E^{n-1}\left[\left|(y-E^{(n-1)}[y_{n-1}])(t_{n,j},x+u)\right|\right]p_{t_{n,j}-t}(u)du\\
		\nonumber  &+C_f\sum_{j=1}^Q w_{n,j}\int\limits_{\mathbb{R}}E^{n-1}\left[\left|(y_{n-1}-E^{(n-1)}[y_{n-1}])(t_{n,j},x+u)\right|\right]p_{t_{n,j}-t}(u)du\\
		\nonumber
		\nonumber &\leq \epsilon_{n}(t)+C_f\sum_{j=1}^Q w_{n,j}\int\limits_{\mathbb{R}}E^{n-1}\left[\left|(y-E^{(n-1)}[y_{n-1}])(t_{n,j},x+u)\right|\right]p_{t_{n,j}-t}(u)du\\
		\nonumber  &+C_f\sum_{j=1}^Q w_{n,j}\int\limits_{\mathbb{R}}(var^{n-1}\left[y_{n-1}(t_{n,j},x+u)\right])^{\frac{1}{2}}p_{t_{n,j}-t}(u)du
		\end{align}
		\noindent where
		\begin{eqnarray}\label{3.6}
		\epsilon_n(t)=\dfrac{[Q!]^4(T-t)^{2Q+1}}{(2Q+1)[(2Q)!]^3}C_d.
		\end{eqnarray}
		From (\ref{3.1}) and Assumption(\ref{ass4}) we have
		\begin{align}\nonumber
		\left|{\int\limits_{t}^T E\left[f(s,y(s,x+W_{s-t}))\right]ds-\sum_{j=1}^Qw_{n,j}E\left[f(t_j, y(t_j,x+W_{t_j-t}))\right]}\right|
		&=\dfrac{[Q!]^4(T-t)^{2Q+1}}{(2Q+1)[(2Q)!]^3}\left|F^{2Q}(\xi)\right|
		\\\nonumber
		&\leq\dfrac{[Q!]^4(T-t)^{2Q+1}}{(2Q+1)[(2Q)!]^3}C_d
		=\epsilon_n(t)
		\end{align}
		Now let us evaluate the variance in each iteration step. From the independence of $(y_{n-1}^i )_{i=1,\ldots,M}$ and the independence of $(W_{t_{n,j}-t}^{n,i} )_{i=1,\ldots,M}$, using the Liptschtz property of $f$ we deduce that
		\begin{align}\label{[align=left]3.7}
		var^n[y_n(t,x)]&=\dfrac{1}{M}var^n[y_{n-1}(t,x)]\\
		\nonumber&+\sum_{j=1}^Q \dfrac{w^2_{n,j}}{M}var^n\left[f(t_{n,j},y_{n-1}(t_{n,j},x+W_{t_{n,j}-t}^{(n)}))-f(t_{n,j},y_{n-2}(t_{n,j},x+W_{t_{n,j}-t}^{(n)}))\right]\\
		\nonumber &\leq\dfrac{1}{M}var^n[y_{n-1}(t,x)]\\
		\nonumber&+\sum_{j=1}^Q \dfrac{w^2_{n,j}}{M}E^n\left[\left(f(t_{n,j},y_{n-1}(t_{n,j},x+W_{t_{n,j}-t}^{(n)}))-f(t_{n,j},y_{n-2}(t_{n,j},x+W_{t_{n,j}-t}^{(n)}))\right)^2\right]\\
		\nonumber&\leq \dfrac{1}{M}var^n[y_{n-1}(t,x)]+\sum_{j=1}^Q \dfrac{C_f^2w_{n,j}^2}{M}E^n[(y_{n-1}-y_{n-2})^2(t_{n,j},x+W_{t_{n,j}-t}^{(n)})]
		\end{align}
		
		Taking squared roots of the both side, it holds that
		\begin{align}\label{3.8} (var^n[y_n(t,x)])^{\frac{1}{2}}&\leq\dfrac{1}{\sqrt{M}}(var^{n-1}[y_n(t,x)])^{\frac{1}{2}}+\dfrac{C_f}{\sqrt{M}}\sum_{j=1}^Q w_{n,j}\left(E^n[(y_{n-1}-y_{n-2})^2(t_{n,j},x+W_{t_{n,j}-t}^{(n)})]\right)^{\frac{1}{2}} 
		\end{align}
		
		On the other hand, from (\ref{2.6}) it holds that

		\begin{align}\label{3.9}
		\left\|(y_n-y_{n-1})(t,x)\right\|_{n,2}&\leq \left\|(\dfrac{1}{M}\sum_{i=1}^M y_{n-1}^i-y_{n-1})(t,x)\right\|_{n,2}\\
		\nonumber &+\sum_{j=1}^Q w_{n,j}\left\|f(t_{n,j},y_{n-1}(t_{n,j},x+W_{t_{n,j}-t}^{(n)}))-f(t_{n,j},y_{n-2}(t_{n,j},x+W_{t_{n,j}-t}^{(n)}))\right\|_{n,2}\\
		\nonumber&\leq\left\|(\dfrac{1}{M}\sum_{i=1}^M y_{n-1}^i-y_{n-1})(t,x)\right\|_{n,2}+ \left\|(y_{n-1}-E^{n-1}y_{n-1})(t,x)\right\|_{n,2}\\
		\nonumber &+C_f\sum_{j=1}^Q w_{n,j}\left\|(y_{n-1}-y_{n-2})(t_{n,j},x+W_{t_{n,j}-t}^{(n)})\right\|_{n,2}\\
		\nonumber &\leq \left(1+\dfrac{1}{\sqrt{M}}\right)(var^{n-1}[y_{n-1}(t,x)])^{\frac{1}{2}}\\\nonumber
		&+C_f\sum_{j=1}^Q w_{n,j}\left\|(y_{n-1}-y_{n-2})(t_{n,j},x+W_{t_{n,j}-t}^{(n)})\right\|_{n,2}
		\end{align}
		
		Let us define 
		\[v_n(t):=\sup_{x\in \mathbb{R}^d}(var^n[y_n(t,x)])^{\frac{1}{2}}, \, \delta_n(t)=\sup_{x\in \mathbb{R}^d}\|(y_n-y_{n-1})(t,x)\|_{n,2}.\] 
		Then for any $n\geq 2$ , from (\ref{3.8}) and (\ref{3.9}) it holds that
		
		\begin{eqnarray}
		v_n(t)\leq \dfrac{1}{\sqrt{M}}v_{n-1}(t)+\dfrac{C_f}{\sqrt{M}}\sum_{j=1}^Q w_{n,j}\delta_{n-1}(t_{n,j}) \label{3.10}
		\end{eqnarray}
		
		\begin{eqnarray}
		\delta_n(t)\leq \left(1+\dfrac{1}{\sqrt{M}}\right)v_{n-1}(t)+C_f\sum_{j=1}^Q w_{n,j}\delta_{n-1}(t_{n,j}) \label{3.11}
		\end{eqnarray}
		
		For the initial step, from the Assumption (\ref{ass1}), Assumption (\ref{ass2}) and (\ref{2.1})-(\ref{2.2}) we deduce that
		
		\begin{eqnarray}
		v_0(t)=0, \,v_1(t)=\sup_{x\in \mathbb{R}^d}\dfrac{\sqrt{var \varphi (x+W_{T-t}^{(1)})}}{\sqrt{M}}\leq \dfrac{C_{\varphi}}{\sqrt{M}} \label{3.12}
		\end{eqnarray}
		
		\begin{eqnarray}
		\delta_1(t)=\sup_{x\in \mathbb{R}^d}\left\|\dfrac{1}{M}\sum_{i=1}^{M}\varphi(x+W_{T-t}^{(1,i)})+\sum_{j=1}^Q w_{1,j}f(t_{1,j},0)\right\|_{1,2}\leq C_{\varphi}+C_0T_1 \label{3.13}
		\end{eqnarray}
		
		Now let us define $a:=1+ \dfrac{1}{\sqrt{M}}$ and substituting (\ref{3.11}) into (\ref{3.10}) repeatedly it holds that
		
		\begin{eqnarray}
		\nonumber v_n(t)&\leq& \dfrac{1}{\sqrt{M}}v_{n-1}(t)+\dfrac{aC_f}{\sqrt{M}}\int\limits_{[t,T]}^{\sim}v_{n-2}(s_1)ds_1+\dfrac{aC_f^2}{\sqrt{M}}\int\limits_{[t,T]}^{\sim}\int\limits_{[s_1,T]}^{\sim}v_{n-3}(s_2)ds_2ds_1\\
		\nonumber \nonumber &&+\cdots+\dfrac{aC_f^{n-2}}{\sqrt{M}}\int\limits_{[t,T]}^{\sim}\int\limits_{[s_1,T]}^{\sim}\cdots \int\limits_{[s_{n-3},T]}^{\sim} v_{1}(s_{n-2})ds_{n-2}ds_{n-3}\cdots ds_1\\
		&&+\dfrac{C_f^{n-1}}{\sqrt{M}}\int\limits_{[t,T]}^{\sim}\int\limits_{[s_1,T]}^{\sim}\cdots \int\limits_{[s_{n-2},T]}^{\sim} \delta_{1}(s_{n-1})ds_{n-1}ds_{n-2}\cdots ds_1 \label{3.14}
		\end{eqnarray}
		
		Because it holds that $\dfrac{d^{2q}}{ds^{2q}}\left(\dfrac{(T-t)^k}{k!}\right)\geq 0$  for any $k,q \in \mathbb{N}$ , from (\ref{3.13}) and Proposition \ref{prop3.1} we deduce that
		
		\begin{eqnarray}
		\nonumber \dfrac{C_f^{n-1}}{\sqrt{M}}\int\limits_{[t,T]}^{\sim}\int\limits_{[s_1,T]}^{\sim}\cdots \int\limits_{[s_{n-2},T]}^{\sim} \delta_{1}(s_{n-1})ds_{n-1}ds_{n-2}\cdots ds_1\leq \dfrac{C_{\varphi}+C_0T}{\sqrt{M}}\cdot\dfrac{C_f^{n-1}(T-t)^{n-1}}{(n-1)!} \\ \label{3.15}
		\end{eqnarray}
		
		For $n\geq 2$, combining (\ref{3.14}) and (\ref{3.15}) it holds that
		
		\begin{eqnarray}
		\nonumber v_n(t)&\leq& \dfrac{a}{\sqrt{M}}\left[v_{n-1}(t)+C_f\int\limits_{[t,T]}^{\sim}v_{n-2}(s_1)ds_1+ C^2_f\int\limits_{[t,T]}^{\sim}\int\limits_{[s_1,T]}^{\sim}v_{n-3}(s_2)ds_2ds_1+\right.\\
		\nonumber &+&\cdots+ C_f^{n-2}\int\limits_{[t,T]}^{\sim}\int\limits_{[s_1,T]}^{\sim}\cdots \int\limits_{[s_{n-3},T]}^{\sim}v_1(s_{n-2})ds_{n-2}ds_{n-3}\cdots ds_{1}\\
		\nonumber &&\left.+(C_{\varphi}+C_0T)\dfrac{C_f^{n-1}(T-t)^{n-1}}{(n-1)!}\right]\\ \label{3.16}
		\end{eqnarray}
		
		Now we prove the following inequality by induction on $2\leq n\in \mathbb{N}$.\\
		
		\begin{align}\label{3.17}\\
		\nonumber v_n(t)\leq \left(\dfrac{a}{\sqrt{M}}\right)^{n-1}\left(\dfrac{C_{\varphi}}{\sqrt{M}}\sum_{k=0}^{n-2}\dfrac{\sqrt{M}^kC_f^k(T-t)^k}{a^kk!}\binom{n-2}{k}+\dfrac{C_{\varphi}+C_0T}{\sqrt{M}}\sum_{k=1}^{n-1}\dfrac{\sqrt{M}^kC_f^k(T-t)^k}{a^kk!}\binom{n-2}{k-1}\right) 
		\end{align}
		
		For the base case where $n=2$, from (\ref{3.10})-(\ref{3.13}) it holds that
		
		$$v_2(t)\leq \dfrac{a}{\sqrt{M}}\dfrac{C_{\varphi}}{\sqrt{M}}+\dfrac{C_f}{\sqrt{M}}(C_{\varphi}+C_0T)(T-t)$$

		Now let us assume that (\ref{3.17}) holds for $v_1(t), \ldots, v_{n-1}(t)$.\\
		
		Then, for any $1\leq l\leq n-1$ the following inequality holds. \\
		\begin{eqnarray*}
			&&C_f^l\int\limits_t^T\int\limits_{s_1}^T\cdots \int\limits_{s_{l-1}}^T v_{n-l-1}(s_l)ds_lds_{l-1}\cdots ds_1 \\
			&\leq& \left(\dfrac{a}{\sqrt{M}}\right)^{n-l-2}\left[\dfrac{C_{\varphi}}{\sqrt{M}}\sum_{k=0}^{n-l-3}\dfrac{\sqrt{M}^kC_f^{k+l}(T-t)^{k+l}}{a^k(k+l)!}\binom{n-l-3}{k}+\right.\\
			&+& \left.\dfrac{C_{\varphi}+C_0T}{\sqrt{M}}\sum_{k=1}^{n-l-2}\dfrac{\sqrt{M}^kC_f^{k+l}(T-t)^{k+l}}{a^k(k+l)!}\binom{n-l-3}{k-1}\right]\\
			&\leq& \left(\dfrac{a}{\sqrt{M}}\right)^{n-2}\left[\dfrac{C_{\varphi}}{\sqrt{M}}\sum_{k=0}^{n-l-3}\dfrac{\sqrt{M}^{k+l}C_f^{k+l}(T-t)^{k+l}}{a^{k+l}(k+l)!}\binom{n-l-3}{k}+\right.\\
			&+& \left.\dfrac{C_{\varphi}+C_0T}{\sqrt{M}}\sum_{k=1}^{n-l-2}\dfrac{\sqrt{M}^{k+l}C_f^{k+l}(T-t)^{k+l}}{a^{k+l}(k+1)!}\binom{n-l-3}{k-1}\right]
		\end{eqnarray*}
		
		If we apply this to each term on the right hand of (\ref{3.17}) then the coefficient of $\dfrac{C_{\varphi}}{\sqrt{M}}(T-t)^k$ from the first part is as follows.
		
		$$\binom{n-3}{k}+\binom{n-4}{k-1}+\cdots+\binom{n-2-k}{1}+\binom{n-3-k}{0}$$
		
		From Pascal's formula we deduce that
		\begin{eqnarray*}
			&&\binom{n-3}{k}+\binom{n-4}{k-1}+\cdots+\binom{n-2-k}{1}+\binom{n-3-k}{0}\\
			&=&\binom{n-3}{k}+\binom{n-4}{k-1}+\cdots+\binom{n-2-k}{1}+\binom{n-2-k}{0}\\
			&=&\binom{n-3}{k}+\binom{n-4}{k-1}+\cdots+\binom{n-1-k}{2}+\binom{n-1-k}{1}\\
			&=&\binom{n-2}{k}
		\end{eqnarray*}
		
		Likewise the coefficient of $\dfrac{C_{\varphi}+C_0T}{\sqrt{M}} (T-t)^k$ from the second part  becomes $\binom{n-2}{k-1}$  and (\ref{3.17}) holds for $n$.\\
		Now from (\ref{3.13}) we have $C_{\varphi}\leq C_{\varphi}+C_0T$ and it holds that		
		\begin{align}\label{3.18}
		v_n(t)&\leq \left(\dfrac{a}{\sqrt{M}}\right)^{n-1}\dfrac{C_{\varphi}+C_0T}{\sqrt{M}}\left(\sum_{k=0}^{n-2}\dfrac{\sqrt{M}^kC_f^k(T-t)^k}{a^kk!}\binom{n-2}{k}+\sum_{k=1}^{n-1}\dfrac{\sqrt{M}^kC_f^k(T-t)^k}{a^kk!}\binom{n-2}{k-1}\right) \\
		\nonumber &\leq \left(\dfrac{a}{\sqrt{M}}\right)^{n-1}\dfrac{C_{\varphi}+C_0T}{\sqrt{M}}\sum_{k=0}^{n-1}\dfrac{\sqrt{M}^kC_f^k(T-t)^k}{a^kk!}\left(\chi(k<n-1)\binom{n-2}{k}+\chi(k>0)\binom{n-2}{k-1}\right)\\
		\nonumber &\leq \left(\dfrac{a}{\sqrt{M}}\right)^{n-1}\dfrac{C_{\varphi}+C_0T}{\sqrt{M}}\sum_{k=0}^{n-1}\dfrac{\sqrt{M}^kC_f^k(T-t)^k}{a^kk!}\binom{n-1}{k}
		\end{align}
		
		Now setting $C_1=\max\{2a, C_{\varphi}+C_0T\}$ , from the Proposition \ref{prop3.3} the following inequality holds.
		
		\begin{eqnarray}\label{3.19}v_n(t)\leq 2^{n-1}\left(\dfrac{a}{\sqrt{M}}\right)^{n-1}\dfrac{C_{\varphi}+C_0T}{\sqrt{M}}e^{C_f\sqrt{M}(T-t)}\leq \left(\dfrac{C_1}{\sqrt{M}}\right)e^{C_f\sqrt{M}(T-t)} 
		\end{eqnarray}
		
		Let us define $\mu_{n}(t):=\sup_{x\in \mathbb{R}^d}\left[|(y-E^n[y_n])(t,x)|\right]$  then from the Assumption \ref{ass3} we have $\mu_0(t)\leq C_y$ and from (\ref{3.5}) we deduce that
		
		\begin{align}\label{3.20}
		\mu_n(t)&\leq \epsilon_n(t)+C_f \int\limits_{[t,T]}^{\sim}\mu_{n-1}(s)ds+C_f\int\limits_{[t,T]}^{\sim}v_{n-1}(s)ds\\ \nonumber &\leq \left(\epsilon_n(t)+C_f\int\limits_{[t,T]}^{\sim}\epsilon_{n-1}(s_1)ds_1\right)+\left(C_f\int\limits_{[t,T]}^{\sim}v_{n-1}(s)ds+C_f^2\int\limits_{[t,T]}^{\sim}\int\limits_{[s_1,T]}^{\sim}v_{n-2}(s_2)ds_2ds_1\right)+\\
		\nonumber &+ C_f^2\int\limits_{[t,T]}^{\sim}\int\limits_{[s_1,T]}^{\sim}\mu_{n-2}(s_2)ds_2ds_1\\
		\nonumber &\leq \left(\epsilon_n(t)+C_f\int\limits_{[t,T]}^{\sim}\epsilon_{n-1}(s_1)ds_1+\cdots+C_f^{n-1}\int\limits_{[t,T]}^{\sim}\int\limits_{[s_1,T]}^{\sim}\cdots \int\limits_{[s_{n-2},T]}^{\sim}\epsilon_{1}(s_{n-1})ds_{n-1}\cdots ds_1\right)+\\
		\nonumber &+\sum_{k=1}^n C_f^k\int\limits_{[t,T]}^{\sim}\int\limits_{[s_1,T]}^{\sim}\cdots \int\limits_{[s_{k-1},T]}^{\sim}v_{n-k}(s_{k})ds_{k}\cdots ds_1\\
		\nonumber &+C_f^n\int\limits_{[t,T]}^{\sim}\int\limits_{[s_1,T]}^{\sim}\cdots \int\limits_{[s_{n-1},T]}^{\sim}\mu_{0}(s_{n})ds_{n}\cdots ds_1
		\end{align}
		
		From (\ref{3.6}) and the Proposition \ref{prop3.1}-\ref{prop3.2}, the first sum on the right hand of (\ref{3.20}) satisfies the following inequality
		
		\begin{align}\label{3.21}\\
		 \nonumber \epsilon_n(t)&+C_f\int\limits_{[t,T]}^{\sim}\epsilon_{n-1}(s_1)ds_1+C_f^{n-1}\int\limits_{[t,T]}^{\sim}\int\limits_{[s_1,T]}^{\sim}\cdots \int\limits_{[s_{n-2},T]}^{\sim}\epsilon_{1}(s_{n-1})ds_{n-1}\cdots ds_1\\
		\nonumber&\leq C_d\dfrac{[Q!]^4}{[(2Q)!]^3}\sum_{k=1}^{n-1}\dfrac{C_f^{k-1}(T-t)^{2Q+k}}{(2Q+1)\cdots(2Q+k)}\\
		\nonumber&\leq C_d \dfrac{e^{\frac{1}{3}}\pi^{\frac{1}{3}}Q^{\frac{1}{2}}}{2}\left(\dfrac{e}{8Q}\right)^{2Q}\sum_{k=1}^{n-1}\dfrac{C_f^{k-1}(T-t)^{2Q+k}}{(2Q+1)\cdots(2Q+k)}\\
		\nonumber&\leq (n-1)Q^{\frac{1}{2}}C_2C_d\left(\dfrac{e}{8Q}\right)^{2Q}\\
		\nonumber&\leq nC_2C_dQ^{\frac{1}{2}}\left(\dfrac{e}{8Q}\right)^{2Q}		
		\end{align}
		\noindent where
		$$C_2=\dfrac{e^{\frac{1}{3}}\pi^{\frac{1}{2}}}{2}\sup_{k\in\mathbb{N}}\dfrac{C_f^{k-1}T^{2q+k}}{(2Q+1)\cdots(2Q+k)}<\infty.$$
		
		From the Taylors expansion of $e^x$  at 0, it holds that
		
		$$\forall n\in \mathbb{N},\exists \xi<x, e^x-\left(1+x+\dfrac{x^2}{2!}+\cdots+\dfrac{x^n}{n!}\right)=\dfrac{e^{\xi}x^{n+1}}{(n+1)!}\leq e^x\dfrac{x^{n+1}}{(n+1)!}$$
		
		From (\ref{3.19}) it holds that
		
		\begin{eqnarray*}
			C_f\int\limits_{[t,T]}^{\sim}v_{n-1}(s_1)ds_1&\leq& C_f\left(\dfrac{C_1}{\sqrt{M}}\right)^{n-1}\int\limits_{t}^Te^{C_f\sqrt{M}(T-s_1)}ds_1\\
			&=& C_f\left(\dfrac{C_1}{\sqrt{M}}\right)^{n-1}\dfrac{e^{C_f\sqrt{M}(T-t)}-1}{\sqrt{M}C_f}\\
			&=& \left(\dfrac{C_1}{\sqrt{M}}\right)^{n}\dfrac{e^{C_f\sqrt{M}(T-t)}-1}{C_1}\\
			&\leq& \left(\dfrac{C_1}{\sqrt{M}}\right)^{n}e^{C_f\sqrt{M}(T-t)}\dfrac{C_f\sqrt{M}(T-t)}{C_1}\\
		\end{eqnarray*}

		\begin{eqnarray*}
			C_f^2\int\limits_{[t,T]}^{\sim}\int\limits_{[s_1,T]}^{\sim}v_{n-2}(s_2)ds_2ds_1&\leq& C_f^2\left(\dfrac{C_1}{\sqrt{M}}\right)^{n-2}\int\limits_{[t,T]}^{\sim}\int\limits_{s_1}^Te^{C_f\sqrt{M}(T-s_2)}ds_2ds_1\\
			&=& C_f^2\left(\dfrac{C_1}{\sqrt{M}}\right)^{n-2}\left(\dfrac{e^{C_f\sqrt{M}(T-t)}-1}{(\sqrt{M}C_f)^2}-\dfrac{T-t}{\sqrt{M}C_f}\right)\\
			&=& \left(\dfrac{C_1}{\sqrt{M}}\right)^{n}\dfrac{e^{C_f\sqrt{M}(T-t)}-1-\sqrt{M}C_f(T-t)}{C_1^2}\\
			&\leq& \left(\dfrac{C_1}{\sqrt{M}}\right)^{n}\dfrac{e^{C_f\sqrt{M}(T-t)}(C_f\sqrt{M}(T-t))^2}{2C_1^2}\\
		\end{eqnarray*}
		
		In a similar way, one can easily check the following inequality by induction on $k\in \mathbb{N}$.		
		\begin{align}\label{3.22}\\
		\nonumber C_f^n\int\limits_{[t,T]}^{\sim}\int\limits_{[s_1,T]}^{\sim}\cdots\int\limits_{[s_{k-1},T]}^{\sim}v_{n-k}(s_k)ds_k\cdots ds_1\leq \left(\dfrac{C_1}{\sqrt{M}}\right)^{n}\dfrac{e^{C_f\sqrt{M}(T-t)}(C_f\sqrt{M}(T-t))^k}{k!C_1^k}
		\end{align}
		
		Summing up (\ref{3.22}) for $k=1,\ldots,n$, it holds that
		
		\begin{align}\label{3.23}\\
		\nonumber \sum_{k=1}^nC_f^n\int\limits_{[t,T]}^{\sim}\int\limits_{[s_1,T]}^{\sim}\cdots \int\limits_{[s_{k-1},T]}^{\sim}v_{n-k}(s_k)ds_k\cdots ds_1&\leq \left(\dfrac{C_1}{\sqrt{M}}\right)^{n}e^{C_f\sqrt{M}(T-t)}\sum_{k=1}^n{\dfrac{(C_f\sqrt{M}(T-t))^k}{k!C_1^k}}\\ \nonumber&\leq \left(\dfrac{C_1}{\sqrt{M}}\right)^{n}e^{C_f\sqrt{M}(T-t)\left(1+\frac{1}{C_1}\right)}
		\end{align}
		
		Likewise the last term of (\ref{3.20}) satisfies the following inequality.
		
		\begin{align}\label{3.24}\\\nonumber
		C_f^n\int\limits_{[t,T]}^{\sim}\int\limits_{[s_1,T]}^{\sim}\cdots\int\limits_{[s_{n-1},T]}^{\sim}\mu_0(s_n)ds_n\cdots ds_1 \leq \dfrac{C_y(T-t)^nC_f^n}{n!}
		\end{align}

		Substituting (\ref{3.21}),(\ref{3.23}),(\ref{3.24}) into (\ref{3.20}), it holds that
		
		\begin{align}\label{3.25}\\\nonumber
		\mu_n(t)\leq nC_2C_dQ^{\frac{1}{2}}\left(\dfrac{e}{8Q}\right)^{2Q}+\left(\dfrac{C_1}{\sqrt{M}}\right)^ne^{C_f\sqrt{M}(T-t)\left(1+\frac{1}{C_1}\right)}+\dfrac{C_y(T-t)^nC_f^n}{n!}
		\end{align}
		(\ref{3.19}) and (\ref{3.25}) proves the result of the theorem.
	\end{proof}
	\bigskip
	
	Note that the main contribution of the error estimates given by the Theorem \ref{the3.1} is from $\left(\dfrac{C_1}{\sqrt{M}}\right)^n$.\\
	Because the complexity of the scheme (\ref{2.4})-(\ref{2.7}) grows at about $d(MQ)^n$, we can see that the complexity grows polynomially in both dimension and error by choosing $M$ and $Q$ properly. (See \cite{1,3} for details)
	
	\section{Conclusion}
	In this paper we proposed a modified multi-level Picard iteration scheme for backward stochastic differential equation and presented an explicit error estimates of the scheme. The proposed scheme is very similar to the original multi-level Picard iteration scheme but it differs slightly and enables improvement in computational complexity. Our further interests will lie on the error estimates of the general case where the generator depends on control variate. Application of multi-level Picard approximation in other fields would also be our interests.

\end{document}